\documentclass{amsproc}
\usepackage{amsmath}
\usepackage{enumerate}
\usepackage{amsmath,amsthm,amscd,amssymb}

\usepackage{latexsym}
\usepackage{upref}
\usepackage{verbatim}

\usepackage[mathscr]{eucal}
\usepackage{dsfont}

\usepackage{graphicx}
\usepackage[colorlinks,hyperindex,pagebackref,hypertex]{hyperref}

\newtheorem{theorem}{Theorem}

\newtheorem{definition}[theorem]{Definition}
\newtheorem{remark}[theorem]{Remark}

\chardef\bslash=`\\ 

\newcommand{\wh}{\widehat}

\newcommand{\dA}{{\dot A}}

\newcommand{\ti}{\tilde  }

\newcommand{\dom}{\text{\rm{Dom}}}

\newcommand{\calH}{{\mathcal H}}

\newcommand{\calR}{{\mathcal R}}

\renewcommand{\Im}{\text{\rm Im}}

   \def\sN{{\mathfrak N}}

\def\bA{{\mathbb A}}      \def\dC{{\mathbb C}}

      \def\dR{{\mathbb R}}

   \def\cH{{\mathcal H}}

\def\RE{{\rm Re\,}}
\def\Ker{{\rm Ker\,}}

\def\wh{\hat}

\def\uphar{{\upharpoonright\,}}

\DeclareMathOperator{\IM}{Im}

\DeclareMathOperator{\Ext}{Ext}

\newcommand{\eval}[2][\right]{\relax
  \ifx#1\right\relax \left.\fi#2#1\rvert}

\begin{document}

\title[Realization of inverse Stieltjes functions $(-m_\alpha(z))$]{ Realization of inverse Stieltjes functions $(-m_\alpha(z))$ \\ by Schr\"odinger L-systems}

\author{S. Belyi}
\address{Department of Mathematics\\ Troy University\\
Troy, AL 36082, USA\\
}
\curraddr{}
\email{sbelyi@troy.edu}


\author{E. Tsekanovski\u i}
\address{Department of Mathematics\\ Niagara University\\
Lewiston, NY 14109\\ USA}
\email{tsekanov@niagara.edu}

\date{DD/MM/2004}



\begin{abstract}
We study L-system realizations of the original Weyl-Titchmarsh functions  $(-m_\alpha(z))$. In the case when the minimal symmetric Schr\"o\-dinger  operator is non-negative, we describe the Schr\"odinger  L-systems that realize  inverse Stieltjes functions $(-m_\alpha(z))$. This approach allows to derive a necessary and sufficient conditions for the functions  $(-m_\alpha(z))$ to be inverse Stieltjes. In particular, the criteria when $(-m_\infty(z))$ is an inverse Stieltjes function is provided. Moreover,  the value  $m_\infty(-0)$ and parameter $\alpha$ allow us to describe the geometric structure of the  realizing  $(-m_\alpha(z))$ L-system.  Additionally, we present the conditions  in terms of the parameter $\alpha$ when the main and associated operators  of a realizing $(-m_\alpha(z))$ L-system have the same or different angle of sectoriality which sets connections with the Kato problem on sectorial extensions of sectorial forms. 

 \end{abstract}

\maketitle

\tableofcontents

\section{Introduction}\label{s1}

The current paper is the third part of the  project (started in \cite{BT18} and continued in \cite{BT19}) that studies the realizations of  the original Weyl-Titchmarsh function $m_\infty(z)$ and its linear-fractional transformation $m_\alpha(z)$ associated with a Schr\"odinger operator.  We  investigate the Herglotz-Nevanlinna functions $-m_\infty(z)$ and $1/m_\infty(z)$ as well as $-m_\alpha(z)$ and $1/m_\alpha(z)$ that  are realized as impedance functions of L-systems containing a dissipative Schr\"o\-dinger main operator $T_h$, ($\IM h>0$). These L-systems  will be refer to  as \textit{Schr\"odinger L-systems} for the rest of the paper.   All formal definitions and expositions of general and Schr\"odinger L-systems are given in Sections \ref{s2} and  \ref{s4}. Note that all Schr\"odinger L-systems $\Theta_{\mu,h}$ form a two-parametric family whose members are uniquely defined by a real-valued parameter $\mu$ and a complex boundary value $h$ ($\IM h>0$) of the main dissipative operator.

In this paper we concentrate on the case when  the realizing Schr\"odinger L-systems are based on non-negative symmetric Schr\"odinger  operator and have accretive main and \textit{accumulative} state-space operator.\footnote{ The situation when the state-space operator  of the realizing Schr\"odinger L-system was accretive was thoroughly considered in \cite{BT19}.} It was shown in \cite{ABT}  (see also \cite{BT-15}) that  the impedance functions of  L-systems with accumulative state-space operators are \textit{inverse Stieltjes} functions. Following our approach from \cite{BT19} here we also set focus on the situation when the realizing accumulative Schr\"odinger L-systems are  sectorial (see Section \ref{s2} for the definition) and the   functions $(-m_\alpha(z))$ are the members of  \textit{sectorial classes} $S^{-1,\beta}$ and $S^{-1,\beta_1,\beta_2}$ of inverse Stieltjes functions that are described in Section \ref{s3}.
Section \ref{s5} is dedicated to  the general realization results from \cite{BT18} for the functions $(-m_\infty(z))$,  $1/m_\infty(z)$, and $(-m_\alpha(z))$. In particular, we recall  there that $(-m_\infty(z))$,  $1/m_\infty(z)$, and $(-m_\alpha(z))$ can be realized as the impedance function of  Schr\"odinger L-systems  $\Theta_{0,i}$, $\Theta_{\infty,i}$, and $\Theta_{\tan\alpha, i}$, respectively.

   Section \ref{s7} contains the main results of the paper when the realization results from Section  \ref{s5} are applied to  Schr\"odinger L-systems with non-negative symmetric Schr\"odinger  operator to obtain important additional properties.  Remark \ref{r-1} of Section   \ref{s7} provides us with the set of criteria for the functions $(-m_\alpha(z))$ to be Stiejtjes or inverse Stijeltjes.  In particular, the Theorem \ref{t-11} and Remark \ref{r-1} give the necessary and sufficient conditions for $(-m_\infty(z))$ to be an inverse Stieltjes function.
   Using the results provided in Section \ref{s4}, we obtain new properties of  L-systems $\Theta_{\tan\alpha, i}$ whose impedance function belong to certain sectorial classes of inverse Stieltjes functions. We emphasize that these results are formulated in terms of the  parameter $\alpha$ defining the function $m_\alpha(z)$.  Also, the knowledge of the limit value  $m_\infty(-0)$ and the value of parameter $\alpha$  lets us  find the exact angles of sectoriality of the main $T_i$ and associate $\ti\bA$ operators of a realizing L-system that establishes the connection to Kato's problem about sectorial extension of sectorial forms.
 
We  conclude the paper with providing an example that illustrates the  main concepts. All the results obtained in this article contribute to a further development of the theory of open physical systems conceived by M.~Liv\u sic in \cite{Lv2}.

\section{Preliminaries}\label{s2}

For a pair of Hilbert spaces $\calH_1$, $\calH_2$ we denote by
$[\calH_1,\calH_2]$ the set of all bounded linear operators from
$\calH_1$ to $\calH_2$. Let $\dA$ be a closed, densely defined,
symmetric operator in a Hilbert space $\calH$ with inner product
$(f,g),f,g\in\calH$. Any non-symmetric operator $T$ in $\cH$ such that
$$\dA\subset T\subset\dA^*$$
is called a \textit{quasi-self-adjoint extension} of $\dA$.

 Consider the rigged Hilbert space (see \cite{Ber}, \cite{ABT})
$\calH_+\subset\calH\subset\calH_- ,$ where $\calH_+ =\dom(\dA^*)$ and
\begin{equation}\label{108}
(f,g)_+ =(f,g)+(\dA^* f, \dA^*g),\;\;f,g \in \dom(A^*).
\end{equation}
Let $\calR$ be the \textit{\textrm{Riesz-Berezansky   operator}} $\calR$ (see \cite{Ber}, \cite{ABT}) which maps $\mathcal H_-$ onto $\mathcal H_+$ such
 that   $(f,g)=(f,\calR g)_+$ ($\forall f\in\calH_+$, $g\in\calH_-$) and
 $\|\calR g\|_+=\| g\|_-$.
 Note that
identifying the space conjugate to $\calH_\pm$ with $\calH_\mp$, we
get that if $\bA\in[\calH_+,\calH_-]$, then
$\bA^*\in[\calH_+,\calH_-].$
An operator $\bA\in[\calH_+,\calH_-]$ is called a \textit{self-adjoint
bi-extension} of a symmetric operator $\dA$ if $\bA=\bA^*$ and $\bA
\supset \dA$.
Let $\bA$ be a self-adjoint
bi-extension of $\dA$ and let the operator $\hat A$ in $\cH$ be defined as follows:
$$
\dom(\hat A)=\{f\in\cH_+:\bA f\in\cH\}, \quad \hat A=\bA\uphar\dom(\hat A).
$$
The operator $\hat A$ is called a \textit{quasi-kernel} of a self-adjoint bi-extension $\bA$ (see \cite{TSh1}, \cite[Section 2.1]{ABT}).
According to the von Neumann Theorem (see \cite[Theorem 1.3.1]{ABT}) the domain of $\wh A$, a self-adjoint extension of $\dA$,  can be expressed as
\begin{equation*}\label{DOMHAT}
\dom(\hat A)=\dom(\dA)\oplus(I+U)\sN_{i},
\end{equation*}
where von Neumann's parameter $U$ is a $(\cdot)$ (and $(+)$)-isometric operator from $\sN_i$ into $\sN_{-i}$  and $$\sN_{\pm i}=\Ker (\dA^*\mp i I)$$ are the deficiency subspaces of $\dA$.

 A self-adjoint bi-extension $\bA$ of a symmetric operator $\dA$ is called \textit{t-self-adjoint} (see \cite[Definition 4.3.1]{ABT}) if its quasi-kernel $\hat A$ is self-adjoint operator in $\calH$.
An operator $\bA\in[\calH_+,\calH_-]$  is called a \textit{quasi-self-adjoint bi-extension} of an operator $T$ if $\bA\supset T\supset \dA$ and $\bA^*\supset T^*\supset\dA.$  

We are mostly interested in the following type of quasi-self-adjoint bi-extensions.
Let $T$ be a quasi-self-adjoint extension of $\dA$ with nonempty resolvent set $\rho(T)$. A quasi-self-adjoint bi-extension $\bA$ of an operator $T$ is called (see \cite[Definition 3.3.5]{ABT}) a \textit{($*$)-extension } of $T$ if $\RE\bA$ is a
t-self-adjoint bi-extension of $\dA$.
In what follows we assume that $\dA$ has deficiency indices $(1,1)$. In this case it is known \cite{ABT} that every  quasi-self-adjoint extension $T$ of $\dA$  admits $(*)$-extensions.
The description of all $(*)$-extensions via Riesz-Berezansky   operator $\calR$ can be found in \cite[Section 4.3]{ABT}.

Recall that a linear operator $T$ in a Hilbert space $\calH$ is called \textbf{accretive} \cite{Ka} if $\RE(Tf,f)\ge 0$ for all $f\in \dom(T)$.  We call an accretive operator $T$
\textbf{$\beta$-sectorial} \cite{Ka} if there exists a value of $\beta\in(0,\pi/2)$ such that
\begin{equation}\label{e8-29}
    (\cot\beta) |\IM(Tf,f)|\le\,\RE(Tf,f),\qquad f\in\dom(T).
\end{equation}
We say that the angle of sectoriality $\beta$ is \textbf{exact} for a $\beta$-sectorial
operator $T$ if $$\tan\beta=\sup_{f\in\dom(T)}\frac{|\IM(Tf,f)|}{\RE(Tf,f)}.$$
An accretive operator is called \textbf{extremal accretive} if it is not $\beta$-sectorial for any $\beta\in(0,\pi/2)$.
 A $(*)$-extension $\bA$ of $T$ is called \textbf{accretive} if $\RE(\bA f,f)\ge 0$ for all $f\in\cH_+$. This is
equivalent to that the real part $\RE\bA=(\bA+\bA^*)/2$ is a nonnegative t-self-adjoint bi-extension of $\dA$.

A ($*$)-extensions $\bA$ of an operator $T$  is called \textbf{accumulative} (see \cite{ABT}) if
\begin{equation}\label{e7-3-3}
(\RE\bA f,f)\le (\dA^\ast f,f)+(f,\dA^\ast f),\quad f\in\calH_+.
\end{equation}

The definition below is a ``lite" version of the definition of L-system given for a scattering L-system with
 one-dimensional input-output space. It is tailored for the case when the symmetric operator of an L-system has deficiency indices $(1,1)$. The general definition of an L-system can be found in \cite[Definition 6.3.4]{ABT} (see also \cite{BHST1} for a non-canonical version).
\begin{definition} 
 An array
\begin{equation}\label{e6-3-2}
\Theta= \begin{pmatrix} \bA&K&\ 1\cr \calH_+ \subset \calH \subset
\calH_-& &\dC\cr \end{pmatrix}
\end{equation}
 is called an \textbf{{L-system}}   if:
\begin{enumerate}
\item[(1)] {$T$ is a dissipative ($\IM(Tf,f)\ge0$, $f\in\dom(T)$) quasi-self-adjoint extension of a symmetric operator $\dA$ with deficiency indices $(1,1)$};
\item[(2)] {$\mathbb  A$ is a   ($\ast $)-extension of  $T$};
\item[(3)] $\IM\bA= KK^*$, where $K\in [\dC,\calH_-]$ and $K^*\in [\calH_+,\dC]$.
\end{enumerate}
\end{definition}
  Operators $T$ and $\bA$ are called a \textit{main and state-space operators respectively} of the system $\Theta$, and $K$ is  a \textit{channel operator}.
It is easy to see that the operator $\bA$ of the system  \eqref{e6-3-2}  is such that $\IM\bA=(\cdot,\chi)\chi$, $\chi\in\calH_-$ and pick $K c=c\cdot\chi$, $c\in\dC$ (see \cite{ABT}).
  A system $\Theta$ in \eqref{e6-3-2} is called \textit{minimal} if the operator $\dA$ is a prime operator in $\calH$, i.e., there exists no non-trivial reducing invariant subspace of $\calH$ on which it induces a self-adjoint operator. Minimal L-systems of the form \eqref{e6-3-2} with  one-dimensional input-output space were also considered in \cite{BMkT}.

We  associate with an L-system $\Theta$ the  function
\begin{equation}\label{e6-3-3}
W_\Theta (z)=I-2iK^\ast (\mathbb  A-zI)^{-1}K,\quad z\in \rho (T),
\end{equation}
which is called the \textbf{transfer  function} of the L-system $\Theta$. We also consider the  function
\begin{equation}\label{e6-3-5}
V_\Theta (z) = K^\ast (\RE\bA - zI)^{-1} K,
\end{equation}
that is called the
\textbf{impedance function} of an L-system $ \Theta $ of the form (\ref{e6-3-2}).  The transfer function $W_\Theta (z)$ of the L-system $\Theta $ and function $V_\Theta (z)$ of the form (\ref{e6-3-5}) are connected by the following relations valid for $\IM z\ne0$, $z\in\rho(T)$,
\begin{equation*}\label{e6-3-6}
\begin{aligned}
V_\Theta (z) &= i [W_\Theta (z) + I]^{-1} [W_\Theta (z) - I],\\
W_\Theta(z)&=(I+iV_\Theta(z))^{-1}(I-iV_\Theta(z)).
\end{aligned}
\end{equation*}
We say that an L-system $\Theta $ of the form \eqref{e6-3-2} is called an \textbf{accretive L-system} (\cite{BT10}, \cite{DoTs}) if its state-space operator  operator $\bA$ is accretive, that is $\RE(\bA f,f)\ge 0$ for all $f\in \calH_+$,
 and \textbf{accumulative} (\cite{BT11}) if its state-space operator $\bA$ is accumulative, i.e., satisfies \eqref{e7-3-3}. It is easy to see that if an L-system is accumulative, then \eqref{e7-3-3} implies that the operator $\dA$ of the system is non-negative and both operators $T$ and $T^*$ are accretive. We also associate another operator $\ti\bA$ to an accumulative L-system $\Theta$. It is given by
\begin{equation}\label{e-14}
    \ti\bA=2\,\RE\dA^*-\bA,
\end{equation}
where $\dA^*$ is in $[\calH_+,\calH_-]$. Obviously, $\RE\dA^*\in[\calH_+,\calH_-]$ and $\ti\bA\in[\calH_+,\calH_-]$. Clearly, $\ti\bA$ is a bi-extension of $\dA$ and is accretive if and only if $\bA$ is accumulative. It is also not hard to see that even though $\ti\bA$ is not a ($*$)-extensions  of the operator $T$ but the form $(\ti\bA f,f)$, $f\in\calH_+$ extends the form $(f,T f)$, $f\in\dom(T)$.
 An accretive  L-system  is called  \textbf{sectorial} if the operator  $\bA$ is sectorial, i.e., satisfies \eqref{e8-29} for some $\beta\in(0,\pi/2)$ and all $f\in\cH_+$. Similarly, an accumulative L-system  is   \textbf{sectorial} if its  operator $\ti\bA$ of the form \eqref{e-14} is sectorial.

\section{Sectorial classes  of inverse Stieltjes functions}\label{s3}

It is known that a scalar function $V(z)$ is called the Herglotz-Nevanlinna function if it is holomorphic on ${\dC \setminus \dR}$, symmetric with respect to the real axis, i.e., $V(z)^*=V(\bar{z})$, $z\in {\dC \setminus \dR}$, and if it satisfies the positivity condition $\IM V(z)\geq 0$,  $z\in \dC_+$.
A complete description of the class of all Herglotz-Nevanlinna functions, that can be realized as impedance functions of L-systems can be found in \cite{ABT}, \cite{BMkT},  \cite{DMTs},  \cite{GT}.
A scalar Herglotz-Nevanlinna function $V(z)$ is a \textit{Stieltjes function} (see \cite{KK74}) if it is holomorphic in $\Ext[0,+\infty)$ and
\begin{equation}\label{e4-0}
\frac{\IM[zV(z)]}{\IM z}\ge0.
\end{equation}
Now we turn to the definition of inverse Stieltjes functions (see \cite{KK74}, \cite{ABT}). 
A scalar Herglotz-Nevanlinna function $V(z)$ is called \textbf{inverse Stieltjes} if $V(z)$ it is holomorphic in $\Ext[0,+\infty)$ and
\begin{equation}\label{e9-187}
\frac{\Im[V(z)/z]}{\IM\, z}\ge0.
\end{equation}
We will consider  the inverse Stieltjes function $V(z)$ that  admit  (see \cite{KK74}) the following integral representation
\begin{equation}\label{e4-1-8}
V(z) =\gamma+\int_{0}^\infty\left( \frac 1{t-z}- \frac{1}{t}\right)\,dG(t),
\end{equation}
where $\gamma\le0$ 
and $G(t)$ is a non-decreasing on $[0,+\infty)$  function such that
$\int^\infty_{0}\frac{dG(t)}{t+t^2}<\infty.$
The following definition  provides the description of a realizable subclass of inverse Stieltjes  functions.
A scalar inverse Stieltjes function $V(z)$ is  a member of the \textbf{class ${S^{-1}_0}(R)$} if  the measure $G(t)$ in  representation \eqref{e4-1-8} is unbounded.
It was shown in \cite[Section 9.9]{ABT} that a function $V(z)$ belongs to the class $S_0^{-1}(R)$ if and only if it can be realized
as impedance function of an accumulative  L-system $\Theta$ of the form \eqref{e6-3-2} with  a non-negative densely defined symmetric operator $\dA$.

The definition of \textbf{sectorial subclasses $S^{-1,\beta}$ } of scalar inverse Stieltjes functions is the  following.
 An inverse Stieltjes function $V(z)$ belongs to $S^{-1,\beta}$ if
\begin{equation}\label{e9-180-i}
K_{\beta}= \sum_{k,l=1}^n\left[\frac{ V(z_k)/z_k-V(\bar z_l)/\bar z_l}{z_k-\bar z_l}-{{(\cot\beta)}~} \frac{V(\bar z_l)}{\bar z_l}\frac{V(z_k)}{z_k}\right]h_k\bar h_l\ge0,
\end{equation}
for an arbitrary sequences of complex numbers $\{z_k\}$, ($\IM z_k>0$) and $\{h_k\}$, ($k=1,...,n$). For
$0<\beta_1< \beta_2 <\frac{\pi}{2}$, we have
\begin{equation*}
S^{-1,\beta_1}\subset S^{-1, \beta_2}\subset{S^{-1}},
\end{equation*}
where $S^{-1}$ denotes the class of all inverse Stieltjes functions (which corresponds to the case $\beta=\frac{\pi}{2}$).

Let $\Theta$ be an accumulative minimal L-system of the form \eqref{e6-3-2}. It was shown in  \cite{BT14} that  the impedance function $V_\Theta(z)$ defined by \eqref{e6-3-5} belongs to the class $S^{-1,\beta}$ if and only if the operator $\ti\bA$ of the form \eqref{e-14} associated to the L-system $\Theta$ is $\beta$-sectorial.

Let $0\le \beta_1<\frac{\pi}{2}$, $0<\beta_2\le \frac{\pi}{2}$, and $\beta_1\le\beta_2$.
We say that a scalar inverse Stieltjes function $V(z)$ of the class $S_0^{-1}(R)$ belongs to the \textbf{class $S^{-1,\beta_1,\beta_2}$} if
\begin{equation}\label{e9-156-i}
    \tan(\pi-\beta_1)=\lim_{x\to0}V(x),\qquad \tan(\pi-\beta_2)=\lim_{x\to-\infty}V(x).
\end{equation}
The following connection between the classes $S^{-1,\beta}$ and $S^{-1,\beta_1,\beta_2}$ was established in \cite{BT14}.
Let $\Theta$ be an accumulative L-system of the form \eqref{e6-3-2} with a densely defined non-negative symmetric operator $\dA$. Let also $\ti\bA$ of the form \eqref{e-14} be  $\beta$-sectorial. Then the impedance function $V_\Theta(z)$ defined by \eqref{e6-3-5} belongs to the class $S^{-1,\beta_1,\beta_2}$. Moreover, the operator $T$ of $\Theta$ is $(\beta_2-\beta_1)$-sectorial 
    with the exact  angle of sectoriality $(\beta_2-\beta_1)$, and $\tan\beta_2\le\tan\beta$.
Note, that this  also remains valid for the case when the operator $\ti \bA$ is accretive but not $\beta$-sectorial for any $\beta\in(0,\pi/2)$.
It also follows that under the same set of assumptions, if   $\beta$ is the exact angle of sectoriality of the operator $T$, then  $V_\Theta(z)\in S^{-1,0,\beta}$ and is such that $\gamma=0$  in \eqref{e4-1-8}.

 Let $\Theta$ be a minimal accumulative L-system of the form \eqref{e6-3-2} as above.  Let also $\ti\bA$ be defined via \eqref{e-14}. It was shown in \cite{BT14} that if  the impedance function $V_\Theta(z)$ belongs to the class  $S^{-1,\beta_1,\beta_2}$  and $\beta_2\ne\pi/2$, then $\ti\bA$ is $\beta$-sectorial, where $\tan\beta$ is defined via
 \begin{equation}\label{e9-176}
    \tan\beta={\tan\beta_2+2\sqrt{\tan\beta_1(\tan\beta_2-\tan\beta_1)}}.
\end{equation}
 Moreover, both $\ti\bA$ and $T$ are $\beta$-sectorial  operators  with  the exact angle $\beta\in(0,\pi/2)$ if and only if  $V_\Theta(z) \in S^{-1,0,\beta}$ and
  \begin{equation}\label{e9-178-new-i}
\tan\beta=\int_0^\infty\frac{dG(t)}{t},
 \end{equation}
 where $G(t)$ is the measure from integral representation \eqref{e4-1-8} of $V_\Theta(z)$  (see  \cite[Theorem 13]{BT14}).



\section{Construction of a Schr\"odinger L-system}\label{s4}

Consider $\calH=L_2[\ell,+\infty)$, $\ell\ge0$, and $l(y)=-y^{\prime\prime}+q(x)y$, where $q$ is a real locally summable on  $[\ell,+\infty)$ function. Suppose that the symmetric operator
\begin{equation}
\label{128}
 \left\{ \begin{array}{l}
 \dA y=-y^{\prime\prime}+q(x)y \\
 y(\ell)=y^{\prime}(\ell)=0 \\
 \end{array} \right.
\end{equation}
has deficiency indices (1,1). Let $D^*$ be the set of functions locally absolutely continuous together with their first derivatives such that $l(y) \in L_2[\ell,+\infty)$. Consider $\calH_+=\dom(\dA^*)=D^*$ with the scalar product
$$(y,z)_+=\int_{\ell}^{\infty}\left(y(x)\overline{z(x)}+l(y)\overline{l(z)}
\right)dx,\;\; y,\;z \in D^*.$$ Let $\calH_+ \subset L_2[\ell,+\infty) \subset \calH_-$ be the corresponding triplet of Hilbert spaces and the operators $T_h$ and $T_h^*$ are
\begin{equation}\label{131}
 \left\{ \begin{array}{l}
 T_hy=l(y)=-y^{\prime\prime}+q(x)y \\
 hy(\ell)-y^{\prime}(\ell)=0 \\
 \end{array} \right.
           ,\quad  \left\{ \begin{array}{l}
 T^*_hy=l(y)=-y^{\prime\prime}+q(x)y \\
 \overline{h}y(\ell)-y^{\prime}(\ell)=0 \\
 \end{array} \right.,
\end{equation}
where $\IM h>0$.
Suppose  $\dA$ is a symmetric operator  of the form \eqref{128} with deficiency indices (1,1), generated by the differential operation $l(y)=-y^{\prime\prime}+q(x)y$. Let also $\varphi_k(x,\lambda) (k=1,2)$ be the solutions of the following Cauchy problems:
$$\left\{ \begin{array}{l}
 l(\varphi_1)=\lambda \varphi_1 \\
 \varphi_1(\ell,\lambda)=0 \\
 \varphi'_1(\ell,\lambda)=1 \\
 \end{array} \right., \qquad
\left\{ \begin{array}{l}
 l(\varphi_2)=\lambda \varphi_2 \\
 \varphi_2(\ell,\lambda)=-1 \\
 \varphi'_2(\ell,\lambda)=0 \\
 \end{array} \right.. $$
It is well known \cite{Na68}, \cite{Levitan} that there exists a function $m_\infty(\lambda)$  introduced by H.~Weyl \cite{W} for which
$$\varphi(x,\lambda)=\varphi_2(x,\lambda)+m_\infty(\lambda)
\varphi_1(x,\lambda)$$ belongs to $L_2[\ell,+\infty)$. It is important for our discussion that the function $m_\infty(\lambda)$ is not a Herglotz-Nevanlinna function but $(-m_\infty(\lambda))$ and $(1/m_\infty(\lambda))$ are (see \cite{Levitan}, \cite{Na68}).

A  construction of an L-system associated with a non-self-adjoint Schr\"odin\-ger operator $T_h$ was thoroughly described in  \cite{ABT}. In particular, it  was shown (see also  \cite{ArTs0})  that  the set of all ($*$)-extensions of the non-self-adjoint Schr\"odinger operator $T_h$ of the form \eqref{131} in $L_2(\ell,+\infty)$ is given by
\begin{equation}\label{137}
\begin{split}
&\bA_{\mu, h}\, y=-y^{\prime\prime}+q(x)y-\frac {1}{\mu-h}\,[y^{\prime}(\ell)-
hy(\ell)]\,[\mu \delta (x-\ell)+\delta^{\prime}(x-\ell)], \\
&\bA^*_{\mu, h}\, y=-y^{\prime\prime}+q(x)y-\frac {1}{\mu-\overline h}\,
[y^{\prime}(\ell)-\overline hy(\ell)]\,[\mu \delta
(x-\ell)+\delta^{\prime}(x-\ell)].
\end{split}
\end{equation}
Note that the formulas \eqref{137} establish a one-to-one correspondence between the set of all ($*$)-extensions of a Schr\"odinger operator $T_h$ of the form \eqref{131} and all real numbers $\mu \in [-\infty,+\infty]$. It is easy to  check that the ($*$)-extension $\bA$ in \eqref{137}  satisfies the condition
\begin{equation*}\label{145}
\IM\bA_{\mu, h}=\frac{\bA_{\mu, h} - \bA^*_{\mu, h}}{2i}=(.,g_{\mu, h})g_{\mu, h},
\end{equation*}
where
\begin{equation}\label{146}
g_{\mu, h}=\frac{(\IM h)^{\frac{1}{2}}}{|\mu - h|}\,[\mu\delta(x-\ell)+\delta^{\prime}(x-\ell)]
\end{equation}
and $\delta(x-\ell), \delta^{\prime}(x-\ell)$ are the delta-function and
its derivative at the point $\ell$, respectively. Furthermore,
\begin{equation*}\label{147}
(y,g_{\mu, h})=\frac{(\IM h)^{\frac{1}{2}}}{|\mu - h|}\ [\mu y(\ell)-y^{\prime}(\ell)],
\end{equation*}
where $y\in \calH_+$, $g_{\mu, h}\in \calH_-$, and $\calH_+ \subset L_2[\ell,+\infty) \subset \calH_-$ is the triplet of Hilbert spaces discussed above.

It was also shown in \cite{ABT} that the quasi-kernel $\hat A_\xi$ of $\RE\bA_{\mu, h}$ is given by
\begin{equation}\label{e-31}
  \left\{ \begin{array}{l}
 \hat A_\xi y=-y^{\prime\prime}+q(x)y \\
 y^{\prime}(\ell)=\xi y(\ell) \\
 \end{array} \right.,\quad \textrm{where} \quad  \xi=\frac{\mu\RE h-|h|^2}{\mu-\RE h}.
\end{equation}
Take operator $K_{\mu, h}{c}=cg_{\mu, h}, \;(c\in \dC)$. Clearly,
\begin{equation}\label{148}
K^*_{\mu, h} y=(y,g_{\mu, h}),\quad  y\in \calH_+,
\end{equation}
and $\IM\bA_{\mu, h}=K_{\mu, h}K^*_{\mu, h}.$ Therefore, 
\begin{equation}\label{149}
\Theta_{\mu, h}= \begin{pmatrix} \bA_{\mu, h}&K_{\mu, h}&1\cr \calH_+ \subset
L_2[\ell,+\infty) \subset \calH_-& &\dC\cr \end{pmatrix},
\end{equation}
is an L-system  with the main operator $T_h$, ($\IM h>0$) of the form \eqref{131},  the state-space operator $\bA_{\mu, h}$ of the form \eqref{137}, and  with the channel operator $K_{\mu, h}$ of the form \eqref{148}. In what follows we will refer to $\Theta_{\mu, h}$ as a \textit{Schr\"odinger L-system}. 
It was established in \cite{ArTs0}, \cite{ABT} that the transfer and impedance functions of $\Theta_{\mu, h}$ are
\begin{equation}\label{150}
W_{\Theta_{\mu, h}}(z)= \frac{\mu -h}{\mu - \overline h}\,\,
\frac{m_\infty(z)+ \overline h}{m_\infty(z)+h},
\end{equation}
and
\begin{equation}\label{1501}
V_{\Theta_{\mu, h}}(z)=\frac{\left(m_\infty(z)+\mu\right)\IM h}{\left(\mu-\RE h\right)m_\infty(z)+\mu\RE h-|h|^2}.
\end{equation}

\section{Schr\"odinger L-system realizations of $-m_\infty(z)$, $1/m_\infty(z)$ and $m_\alpha(z)$}\label{s5}
 
 As we have already mentioned in Section \ref{s4},  the original  Weyl-Titchmarsh function $m_\infty(z)$ has a property that $(-m_\infty(z))$ is a Herglotz-Nevanlinna function (see  \cite{Levitan}, \cite{Na68}). A problem whether $(-m_\infty(z))$ can be realized as the impedance function of a Schr\"odinger L-system was solved  in the following theorem  proved in \cite{BT18}.

\begin{theorem}[\cite{BT18}]\label{t-6}%
Let $\dA$ be a  symmetric Schr\"odinger operator of the form \eqref{128} with deficiency indices $(1, 1)$ and locally summable potential in $\calH=L^2[\ell, \infty).$ If $m_\infty(z)$ is the  Weyl-Titchmarsh function of $\dA$, then the Herglotz-Nevanlinna function $(-m_\infty(z))$ can be realized as the impedance function of  a Schr\"odinger L-system $\Theta_{\mu, h}$ of the form \eqref{149} with $\mu=0$ and $h=i$.

Conversely, let $\Theta_{\mu, h}$ be  a Schr\"odinger L-system of the form \eqref{149} with the symmetric operator $\dA$ such that $V_{\Theta_{\mu, h}}(z)=-m_\infty(z),$ for all $z\in\dC_\pm$ and $\mu\in\mathbb R\cup\{\infty\}$. Then the parameters $\mu$ and $h$ defining $\Theta_{\mu, h}$ are such that $\mu=0$ and $h=i$.
\end{theorem}

An analogues result for the function $1/m_\infty(z)$ also takes place (see \cite{BT18}).
\begin{theorem}[\cite{BT18}]\label{t-7}%
Let $\dA$ be a  symmetric Schr\"odinger operator of the form \eqref{128} with deficiency indices $(1, 1)$ and locally summable potential in $\calH=L^2[\ell, \infty).$ If $m_\infty(z)$ is the  Weyl-Titchmarsh function of $\dA$, then the Herglotz-Nevanlinna function $(1/m_\infty(z))$ can be realized as the impedance function of  a Schr\"odinger L-system $\Theta_{\mu, h}$ of the form \eqref{149} with $\mu=\infty$ and $h=i$.

Conversely, let $\Theta_{\mu, h}$ be  a Schr\"odinger L-system of the form \eqref{149} with the symmetric operator $\dA$ such that $V_{\Theta_{\mu, h}}(z)=\frac{1}{m_\infty(z)},$ for all $z\in\dC_\pm$ and $\mu\in\mathbb R\cup\{\infty\}$. Then the parameters $\mu$ and $h$ defining $\Theta_{\mu, h}$ are such that $\mu=\infty$ and $h=i$.
\end{theorem}
One can note that  both L-systems $\Theta_{0,i}$  and $\Theta_{\infty,i}$ obtained in Theorems \ref{t-6} and \ref{t-7} share the same main operator
\begin{equation}\label{e-56-T}
    \left\{ \begin{array}{l}
 T_{i}\, y=-y^{\prime\prime}+q(x)y \\
 y'(\ell)=i\,y(\ell) \\
 \end{array} \right..
\end{equation}

The Weyl-Titchmarsh functions $m_\alpha(z)$ are defined as follows.  Let  $\dA$ be a symmetric operator  of the form \eqref{128} with deficiency indices (1,1), generated by the differential operation $l(y)=-y^{\prime\prime}+q(x)y$. Let also $\varphi_\alpha(x,{z})$ and $\theta_\alpha(x,{z})$ be the solutions of the following Cauchy problems:
$$\left\{ \begin{array}{l}
 l(\varphi_\alpha)={z} \varphi_\alpha \\
 \varphi_\alpha(\ell,{z})=\sin\alpha \\
 \varphi'_\alpha(\ell,{z})=-\cos\alpha \\
 \end{array} \right., \qquad
\left\{ \begin{array}{l}
 l(\theta_\alpha)={z} \theta_\alpha \\
 \theta_\alpha(\ell,{z})=\cos\alpha \\
 \theta'_\alpha(\ell,{z})=\sin\alpha \\
 \end{array} \right.. $$
One can show  \cite{DanLev90}, \cite{Na68}, \cite{Ti62} that there exists an analytic in $\dC_\pm$  function $m_\alpha({z})$  for which
\begin{equation}\label{e-62-psi}
\psi(x,{z})=\theta_\alpha(x,{z})+m_\alpha({z})\varphi_\alpha(x,{z})
\end{equation}
belongs to $L_2[\ell,+\infty)$. It is easy to see that if $\alpha=\pi$, then $m_\pi({z})=m_\infty({z})$. The functions $ m_\alpha({z})$ and $m_\infty(z)$ are connected (see \cite{DanLev90}, \cite{Ti62}) by
\begin{equation}\label{e-59-LFT}
    m_\alpha({z})=\frac{\sin\alpha+m_\infty({z})\cos\alpha}{\cos\alpha-m_\infty({z})\sin\alpha}.
\end{equation}
It is known \cite{Na68}, \cite{Ti62} that for any real $\alpha$ the function $-m_\alpha({z})$ is a Herglotz-Nevanlinna function. Also,  \eqref{e-59-LFT} yields
\begin{equation}\label{e-61-Don}
    -m_\alpha(z)=\frac{\sin\alpha+m_\infty(z)\cos\alpha}{-\cos\alpha+m_\infty(z)\sin\alpha}
    =\frac{\cos\alpha+\frac{1}{m_\infty(z)}\sin\alpha}{\sin\alpha-\frac{1}{m_\infty(z)}\cos\alpha}.
\end{equation}
The theorem below was proved in  \cite{BT18} for Herglotz-Nevanlinna functions $-m_\alpha(z)$  and  is similar to Theorem \ref{t-6}.
\begin{theorem}[\cite{BT18}]\label{t-8}%
Let $\dA$ be a  symmetric Schr\"odinger operator of the form \eqref{128} with deficiency indices $(1, 1)$ and locally summable potential in $\calH=L^2[\ell, \infty).$ If $m_\alpha(z)$ is the  function of $\dA$ described in \eqref{e-62-psi}, then the Herglotz-Nevanlinna function $(-m_\alpha(z))$ can be realized as the impedance function of  a Schr\"odinger L-system $\Theta_{\mu, h}$ of the form \eqref{149} with
\begin{equation}\label{e-62-h-mu}
    \mu=\tan\alpha\quad \textrm{and}\quad h=i.
\end{equation}

Conversely, let $\Theta_{\mu, h}$ be  a Schr\"odinger L-system of the form \eqref{149} with the symmetric operator $\dA$ such that $$V_{\Theta_{\mu, h}}(z)=-m_\alpha(z),$$ for all $z\in\dC_\pm$ and $\mu\in\mathbb R\cup\{\infty\}$. Then the parameters $\mu$ and $h$ defining $\Theta_{\mu, h}$ are given by \eqref{e-62-h-mu}, i.e., $\mu=\tan\alpha$ and $h=i$.
\end{theorem}

Clearly, when $\alpha=\pi$ we obtain $\mu_\alpha=0$, $m_\pi(z)=m_\infty(z)$, and the realizing  L-system $\Theta_{0,i}$ is thoroughly described  in \cite[Section 5]{BT18}. If $\alpha=\pi/2$, then  we get $\mu_\alpha=\infty$, $-m_\alpha(z)=1/m_\infty(z)$, and the realizing  L-system is $\Theta_{\infty,i}$ (see \cite[Section 5]{BT18}). 
Excluding the cases when $\alpha=\pi$ or $\alpha=\pi/2$, we give the description of a Schr\"odinger L-system $\Theta_{\mu_\alpha,i}$ realizing $-m_\alpha(z)$ for $\alpha\in(0,\pi]$ as follows
\begin{equation}\label{e-64-sys}
    \Theta_{\tan\alpha, i}= \begin{pmatrix} \bA_{\tan\alpha, i}&K_{\tan\alpha, i}&1\cr \calH_+ \subset
L_2[\ell,+\infty) \subset \calH_-& &\dC\cr \end{pmatrix},
\end{equation}
where
\begin{equation}\label{e-65-star}
\begin{split}
&\bA_{\tan\alpha,i}\, y=l(y)-\frac{1}{\tan\alpha-i}[y^{\prime}(\ell)-iy(\ell)][(\tan\alpha)\delta(x-\ell)+\delta^{\prime}(x-\ell)], \\
&\bA^*_{\tan\alpha,i}\, y=l(y)-\frac{1}{\tan\alpha+i}\,[y^{\prime}(\ell)+iy(\ell)][(\tan\alpha)\delta(x-\ell)+\delta^{\prime}(x-\ell)],
\end{split}
\end{equation}
$K_{\tan\alpha, i}\,{c}=c\,g_{\tan\alpha, i}$, $(c\in \dC)$ and
\begin{equation}\label{e-66-g}
g_{\tan\alpha, i}=(\tan\alpha)\delta(x-\ell)+\delta^{\prime}(x-\ell).
\end{equation}
It is also worth mentioning that
\begin{equation}\label{e-71-VW}
\begin{aligned}
    V_{\Theta_{\tan\alpha, i}}(z)&=-m_\alpha(z)\\
     W_{\Theta_{\tan\alpha, i}}(z)&=\frac{\tan\alpha-i}{\tan\alpha+i}\cdot\frac{m_\infty(z)-i}{m_\infty(z)+i}=(-e^{2\alpha i})\,\frac{m_\infty(z)-i}{m_\infty(z)+i}.
    \end{aligned}
\end{equation}

Similar to Theorem \ref{t-7} results  for the functions $1/m_\alpha(z)$ can be found in \cite{BT18}.


\section{Accumulative Schr\"odinger L-systems}\label{s7}

In this section we assume  that  $\dA$ is a \textit{non-negative} (i.e., $(\dA f,f) \geq 0$ for all $f \in \dom(\dA)$) symmetric operator  of the form \eqref{128} with deficiency indices (1,1), generated by the differential operation $l(y)=-y^{\prime\prime}+q(x)y$. The following theorem takes place.
\begin{theorem}[\cite{T87},  \cite{Ts81}, \cite{Ts80}]\label{t-10}
Let $\dA$ be a nonnegative symmetric Schr\"odinger operator of the form \eqref{128} with deficiency indices $(1, 1)$ and locally summable potential in $\calH=L^2[\ell,\infty).$ Consider operator $T_h$ of the form \eqref{131}.  Then
 \begin{enumerate}
\item operator $\dA$ has more than one non-negative self-adjoint extension, i.e., the Friedrichs extension $A_F$ and the Kre\u{\i}n-von Neumann extension $A_K$ do not coincide, if and only if $m_{\infty}(-0)<\infty$;
 \item operator $T_h$, ($h=\bar h$) coincides with the Kre\u{\i}n-von Neumann extension $A_K$ if and  only if $h=-m_{\infty}(-0)$;
\item operator $T_h$ is accretive if and only if
\begin{equation}\label{138}
\RE h\geq-m_\infty(-0);
\end{equation}
\item operator $T_h$, ($h\ne\bar h$) is $\beta$-sectorial if and only if  $\RE h >-m_{\infty}(-0)$ holds;
\item operator $T_h$, ($h\ne\bar h$) is accretive but not $\beta$-sectorial for any $\beta\in (0, \frac{\pi}{2})$ if and only if $\RE h=-m_{\infty}(-0)$ \item If $T_h, (\IM h>0)$ is $\beta$-sectorial,
then the exact angle $\beta$ can be calculated via
\begin{equation}\label{e10-45}
\tan\beta=\frac{\IM h}{\RE h+m_{\infty}(-0)}.
\end{equation}
\end{enumerate}
\end{theorem}
In what follows,  we assume that $m_{\infty}(-0)<\infty$. Then according to Theorem \ref{t-10}  (see also \cite{AT2009}, \cite{Ts2}, \cite{Ts80}) the operator $T_h$, ($\IM h>0$) of the form \eqref{131} is accretive and/or sectorial.
 If in this case $T_h$ is  accretive, then (see \cite{ABT}) for all real $\mu$ satisfying the  inequality
\begin{equation}\label{151}
\mu \geq \frac {(\IM h)^2}{m_\infty(-0)+\RE h}+\RE h,
\end{equation}
formulas \eqref{137} define the set of all accretive $(*)$-extensions $\bA_{\mu, h}$ of  $T_h$. Moreover, $\bA_{\mu, h}$ is accretive but not  $\beta$-sectorial for any $\beta\in(0,\pi/2)$ ($*$)-extension of $T_h$
 if and only if in \eqref{137}
\begin{equation}\label{e10-134}
\mu=\frac{(\IM h)^2}{m_\infty(-0)+\RE h}+\RE h,
\end{equation}
 (see \cite[Theorem 4]{BT-15}). It is also shown in \cite{ABT} that $(*)$-extensions $\bA_{\mu, h}$ of the operator $T_h$ are accumulative  if and only if
\begin{equation}\label{250}
-m_\infty(-0) \leq \mu \leq \RE h.
\end{equation}
Using formulas \eqref{137} and direct calculations (see also \cite{BT-15}) one can obtain the formula for operator $\ti\bA_{\mu, h}$ of the form \eqref{e-14} as follows
\begin{equation}\label{e-27}
    \begin{aligned}
    \ti\bA_{\mu, h} y=-y^{\prime\prime}&+q(x)y-y'(a)\delta(x-a)-y(a)\delta'(x-a)\\
   & +\frac {1}{\mu-h}\,[y'(a)-hy(a)]\,[\mu \delta (x-a)+\delta'(x-a)].
    \end{aligned}
\end{equation}

Consider the functions $m_\alpha(z)$ described by \eqref{e-62-psi}-\eqref{e-59-LFT} and associated with the non-negative operator $\dA$ above. Let us observe how the parameter $\alpha$ in the definition of $m_\alpha(z)$ effects the  L-system realizing $(-m_\alpha(z))$. Part of this  question was answered in \cite[Theorem 6.3]{BT18}. It was shown that if the non-negative symmetric Schr\"odinger operator is such that $m_{\infty}(-0)\ge0$, then the L-system $\Theta_{\tan\alpha, i}$ of the form \eqref{e-64-sys} realizing the  function $(-m_\alpha(z))$ is accretive if and only if
\begin{equation}\label{e-78-angles}
\tan\alpha\ge \frac{1}{m_{\infty}(-0)}.
\end{equation}
We are going to use inequality \eqref{250} to see the values of $\mu=\tan\alpha$ that generate accumulative  L-systems $\Theta_{\tan\alpha, i}$. This approach yields
\begin{equation}\label{e-39-ac-angles}
-m_\infty(-0) \leq\tan\alpha\leq 0.
\end{equation}
The established criteria for a function $(-m_\alpha(z))$ to be realized with an accretive or accumulative L-system $\Theta_{\tan\alpha, i}$ are graphically shown on Figure \ref{fig-1}. This figure describes the dependence of the properties of realizing $(-m_\alpha(z))$ L-systems on the value of $\mu$ and hence $\alpha$. The bold part of the real line depicts values of $\mu=\tan\alpha$ that produce accretive or accumulative L-systems $\Theta_{\mu, i}$.

\begin{figure}
  \begin{center}
  \includegraphics[width=120mm]{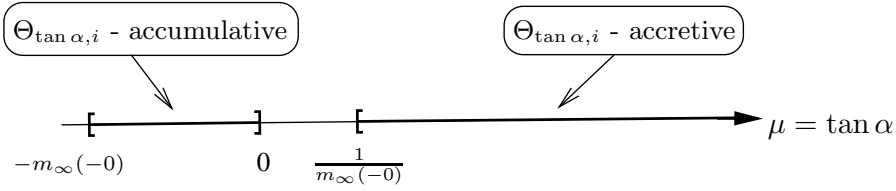}
  \caption{Accumulative and accretive L-systems $\Theta_{\tan\alpha, i}$.}\label{fig-1}
  \end{center}
\end{figure}

Note that if $m_\infty(-0)=0$ in \eqref{e-78-angles}, then $\alpha=\pi/2$ and $-m_{\frac{\pi}{2}}(z)={1}/{m_\infty(z)}$. Moreover,   we know that  if $m_\infty(-0)\ge0$, then ${1}/{m_\infty(z)}$ is realized by an accretive system $\Theta_{\infty, i}$ (see  \cite[Theorem 6.2]{BT18}). We also note that when $\tan\alpha=0$ and hence $\alpha=0$ we obtain  $m_0(z)=m_\infty(z)$, and the realizing $-m_\infty(z)$ Schr\"odinger L-system is $\Theta_{0,i}$. The following theorem shows  how the additional requirement of non-negativity affects the realization of functions $-m_\infty(z)$ and $1/m_\infty(z)$.
\begin{theorem}\label{t-11}%
Let $\dA$ be a non-negative symmetric Schr\"odinger operator of the form \eqref{128} with deficiency indices $(1, 1)$ and locally summable potential in $\calH=L^2[\ell, \infty).$ If $m_\infty(z)$ is the  Weyl-Titchmarsh function of $\dA$ such that $m_\infty(-0)\ge0$, then the L-system $\Theta_{0, i}$  realizing the  function $(-m_\infty(z))$ is accumulative and the L-system $\Theta_{\infty, i}$  realizing the  function $1/m_\infty(z)$ is accretive.
 \end{theorem}
\begin{proof}
Since $m_\infty(-0)\ge0$, we can apply  \eqref{e-39-ac-angles} to conclude that $-m_0(z)=-m_\infty(z)\le0$ implies that the L-system $\Theta_{0, i}$  realizing the  function $(-m_\infty(z))$ is accumulative (see \cite[Section 9.9]{ABT}). The fact that the L-system $\Theta_{\infty, i}$  realizing the  function $1/m_\infty(z)$ is accretive under the conditions of current theorem was proved in \cite{BT18}.
\end{proof}
\begin{remark}\label{r-1}
 Some of analytic properties of the functions $(-m_\infty(z))$, $1/m_\infty(z)$, and $(-m_\alpha(z))$ were described in \cite[Theorem 6.5]{BT18}. Taking into account these results and  the above reasoning we have that under the current set of assumptions:
 \begin{enumerate}
   \item the  function $1/m_\infty(z)$ is Stieltjes if and only if $m_\infty(-0)\ge0$;
   \item the  function $(-m_\infty(z))$ is inverse Stieltjes  if and only if $m_\infty(-0)\ge0$;
   \item the  function $(-m_\alpha(z))$ given by \eqref{e-59-LFT} is Stieltjes if and only if $$0<\frac{1}{m_\infty(-0)}\le\tan\alpha,$$
   and inverse Stieltjes if and only if $$-m_\infty(-0)\le\tan\alpha\le0.$$
 \end{enumerate}
\end{remark}

Now once we established a criteria for an L-system realizing $(-m_\alpha(z))$ to be accumulative, we can look into more of its properties.
We are going to turn to the case when our realizing L-system $\Theta_{\tan\alpha, i}$ is accumulative sectorial.
To begin with let  $\Theta_{\mu, h}$ be an L-system of the form \eqref{149}, where $\bA_{\mu, h}$ is an accumulative ($*$)-extension \eqref{137} of the accretive Schr\"odinger operator $T_h$. Let also $\ti\bA_{\mu, h}$ be of the form \eqref{e-27}. Below  is the list of some known facts about possible accumulativity and  sectoriality of $\Theta_{\mu, h}$.
\begin{itemize}
 \item If $\ti\bA_{\mu, h}$ of the form \eqref{e-27} is  $\beta$-sectorial, then the impedance function $V_{\Theta_{\mu, h}}(z)$ defined by \eqref{e6-3-5} belongs to the class $S^{-1,\beta_1,\beta_2}$.
  \item The operator $T_h$ of $\Theta_{\mu, h}$ is $(\beta_2-\beta_1)$-sectorial     with the exact  angle of sectoriality $(\beta_2-\beta_1)$, and $\tan\beta_2\le\tan\beta$.
  \item In the case when $\beta_1=0$ and $\beta_2=\pi/2$ the operator $T_h$ is accretive but not $\beta$-sectorial.
  \item If   $\beta$ is the exact angle of sectoriality of the operator $T_h$, then  $V_{\Theta_{\mu, h}}(z)\in S^{-1,0,\beta}$.
    \item if  the impedance function $V_{\Theta_{\mu, h}}(z)$ belongs to the class  $S^{-1,\beta_1,\beta_2}$, then $\ti\bA_{\mu, h}$ is $\beta$-sectorial, where $\tan\beta$ is defined via \eqref{e9-176}.
  \item   Both $\ti\bA_{\mu, h}$ and $T_h$ are $\beta$-sectorial  operators  with  the exact angle $\beta\in(0,\pi/2)$ if and only if  $V_{\Theta_{\mu, h}}(z) \in S^{-1,0,\beta}$ and $\tan\beta$ is given by \eqref{e9-178-new-i}.
\end{itemize}

Consider a function $(-m_\alpha(z))$ and Schr\"odinger L-system $\Theta_{\tan\alpha, i}$ of the form \eqref{e-64-sys} that realizes it. According to Theorem \ref{t-11} this L-system $\Theta_{\tan\alpha, i}$ can be accumulative if and only if \eqref{e-39-ac-angles} holds, that is $-m_\infty(-0)\le\tan\alpha\le0.$ Moreover, according to \cite[Theorem 6]{BT-15}, $\Theta_{\tan\alpha, i}$ is accumulative sectorial if and only if
\begin{equation}\label{e-45}
-m_\infty(-0)\le\tan\alpha<0,
\end{equation}
and accumulative extremal (see \cite[Theorem 7]{BT-15}) if and only if $\tan\alpha=0$. Also, if we assume that L-system $\Theta_{\tan\alpha, i}$ is $\beta$-sectorial, then its impedance function \break $V_{\Theta_{\tan\alpha, i}}(z)=-m_\alpha(z)$ belongs (see \cite{BT14}) to certain sectorial classes of inverse Stieltjes functions  discussed in Section \ref{s3}. Namely, $(-m_\alpha(z))\in S^{-1,\beta}$. The following theorem provides more refined properties of $(-m_\alpha(z))$ for this case.
\begin{theorem}\label{t-15}
Let $\Theta_{\tan\alpha, i}$ be the accumulative L-system of the form \eqref{e-64-sys} realizing the  function  $(-m_\alpha(z))$  associated with the non-negative operator $\dA$. Let also $\ti\bA_{\tan\alpha,i}$ be a $\beta$-sectorial operator associated with $\Theta_{\tan\alpha, i}$ and  defined by \eqref{e-14}. Then the  function $(-m_\alpha(z))$  belongs to the class $S^{-1,\beta_1,\beta_2}$, $\tan\beta_2\le\tan\beta$, and
\begin{equation}\label{e-46}
    \tan\beta_1=\frac{\tan\alpha+m_\infty(-0)}{1-(\tan\alpha) m_\infty(-0)},
\end{equation}
and
\begin{equation}\label{e-47}
    \tan\beta_2=-\cot\alpha.
\end{equation}
Moreover, the operator $T_i$ is $(\beta_2-\beta_1)$-sectorial with the exact  angle of sectoriality $(\beta_2-\beta_1)$.
 \end{theorem}
 \begin{proof}
It is given that $\Theta_{\tan\alpha, i}$ is accumulative and hence \eqref{e-45} holds. For further convenience we re-write $(-m_\alpha(z))$ as
\begin{equation}\label{e-48}
    -m_\alpha(z)=\frac{\sin\alpha+m_\infty(z)\cos\alpha}{-\cos\alpha+m_\infty(z)\sin\alpha}
    =\frac{\tan\alpha+m_\infty(z)}{(\tan\alpha) m_\infty(z)-1}.
\end{equation}
Since under our assumption  $\ti\bA_{\tan\alpha,i}$ is $\beta$-sectorial, then (see \cite{BT14}, \cite{BT-15}) the impedance function $V_{\Theta_{\tan\alpha, i}}(z)=-m_\alpha(z)$ belongs to certain sectorial classes discussed in Section \ref{s3}. Particularly, $-m_\alpha(z)\in S^{-1,\beta}$ and $-m_\alpha(z)\in S^{-1,\beta_1,\beta_2}$, where (see \cite{BT-15})
\begin{equation*}
     \tan(\pi-\beta_1)= - \tan\beta_1=\lim_{x\to-0}(-m_\alpha(x))=\frac{\tan\alpha+m_\infty(-0)}{(\tan\alpha) m_\infty(-0)-1},
\end{equation*}
and
\begin{equation*}
 \begin{aligned}
 \tan(\pi-\beta_2)&= -\tan\beta_2=\lim_{x\to-\infty}(-m_\alpha(x))=\frac{\tan\alpha+m_\infty(-\infty)}{(\tan\alpha) m_\infty(-\infty)-1}\\
 &=\frac{\frac{\tan\alpha}{m_\infty(-\infty)}+1}{\tan\alpha -\frac{1}{m_\infty(-\infty)}}=\frac{1}{\tan\alpha}=\cot\alpha.
    \end{aligned}
\end{equation*}
Multiplying the above by $(-1)$ one confirms \eqref{e-46} and \eqref{e-47}. In order to show the rest, we apply \cite[Theorem 9]{BT14}. This theorem states that if  $\ti\bA$ is a $\beta$-sectorial operator of the form \eqref{e6-3-5} associated to an accumulative L-system $\Theta$, then the impedance function $V_\Theta(z)$ belongs to the class $S^{-1,\beta_1,\beta_2}$, $\tan\beta_2\le\tan\beta$, and  $T$ is $(\beta_2-\beta_1)$-sectorial with the exact  angle of sectoriality $(\beta_2-\beta_1)$.
 \end{proof}

\begin{figure}
  \begin{center}
  \includegraphics[width=120mm]{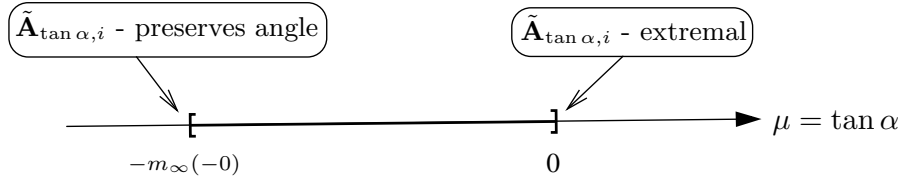}
  \caption{Associated operator $\ti{\bA}_{\tan\alpha, i}$.}\label{fig-2}
  \end{center}
\end{figure}

The next theorem explains two ``endpoint" cases of accumulative realization for the function  $(-m_\alpha(z))$.

\begin{theorem}\label{t10-17-i}
Let $\Theta_{\tan\alpha, i}$ be the accumulative L-system of the form \eqref{e-64-sys} realizing the  function  $(-m_\alpha(z))$   with a sectorial main operator $T_i$ whose  exact angle of sectoriality is $\beta\in(0,\pi/2)$.  Let also $\ti\bA_{\tan\alpha,i}$ be an associated operator defined by \eqref{e6-3-5}.  Then
\begin{enumerate}
  \item $\ti\bA_{\tan\alpha,i}$ is  $\beta$-sectorial (with the same angle of sectoriality as $T_i$) if and only if $\tan\alpha=-m_\infty(-0)$ in \eqref{e-27};
  \item $\ti\bA_{\tan\alpha,i}$ is accretive but not  $\beta$-sectorial for any $\beta\in(0,\pi/2)$
 if and only if in \eqref{137} $\alpha=0.$
\end{enumerate}
\end{theorem}
\begin{proof}
The proof directly follows from \cite[Theorems 6 and 7]{BT-15} after one sets $\mu=\tan\alpha=-m_\infty(-0)$ for part (1) and $\mu=\RE h=\tan 0=0$ for part (2).
 \end{proof}
The result of Theorem \ref{t10-17-i} is graphically illustrated by Figure \ref{fig-2}. Also we have shown that within the conditions of Theorem \ref{t10-17-i} the $\alpha$-sectorial sesquilinear  form $(f,T f)$ defined on a subspace $\dom(T)$ of $\calH_+$ can be extended to the $\alpha$-sectorial form $(\ti\bA f,f)$ defined on $\calH_+$ preserving the exact (for both forms) angle of sectoriality $\alpha$. A general problem of extending sectorial sesquilinear forms was mentioned by T.~Kato in \cite{Ka}.

Now we state and prove the following.
\begin{theorem}\label{t-16}
Let $\Theta_{\tan\alpha, i}$ be an accumulative  L-system of the form \eqref{e-64-sys} that realizes $(-m_\alpha(z))$ with the main   $\theta$-sectorial operator $T_i$ whose exact sectoriality angle   is $\theta$. Let also  $\alpha_*\in\left(\arctan(-m_{\infty}(-0)),0\right)$ be a fixed value that defines the associated operator $\ti\bA_{\tan\alpha_*,i}$ via \eqref{e6-3-5},  \eqref{137},   and  $(-m_\alpha(z))\in S^{-1,\beta_1,\beta_2}$.
Then the associated operator $\ti\bA_{\tan\alpha,i}$  is $\beta$-sectorial for any $\alpha\in \left(\arctan(-m_{\infty}(-0)),\alpha_*\right)$ with
\begin{equation}\label{e6-6}
 {\tan\beta=\tan\beta_1+2\sqrt{\tan\beta_1\,\tan\beta_2}}.
\end{equation}
Moreover, if $\alpha=\arctan(-m_{\infty}(-0))$, then $$\beta=\theta=\arctan\left(\frac{1}{m_{\infty}(-0)}\right).$$
\end{theorem}
\begin{proof}
We note first that the conditions of our theorem imply the following: $\tan\alpha_*\in\left(-m_{\infty}(-0),0\right)$. Thus, according to \cite[Theorem 8]{BT18} applied for $\mu=\tan\alpha$ the operator $\ti\bA_{\tan\alpha,i}$ is $\beta$-sectorial for some $\beta\in(0,\pi/2)$ for any $\alpha$ such that
$$-m_{\infty}(-0)\le\tan\alpha<\tan\alpha_*.$$
Formula \eqref{e6-6} also follows from the corresponding formula in \cite[Theorem 8]{BT18} taken into account that $\beta_1$ and $\beta_2$ are defined via \eqref{e-46} and \eqref{e-47}, respectively.
 Finally, since $T_i$ is $\theta$-sectorial, formula \eqref{e10-45} yields $\tan\theta=\frac{1}{m_{\infty}(-0)}$. Applying part (1) of Theorem \ref{t10-17-i} gives us that $\beta=\theta$. This completes the proof.
 \end{proof}
Note that Theorem \ref{t-16} provides us with a value $\beta$ which serves as a universal angle of sectoriality for the entire indexed family of   associated operators $\ti\bA$ of the form \eqref{e-27} as depicted on Figure \ref{fig-3}.

\begin{figure}
  \begin{center}
  \includegraphics[width=90mm]{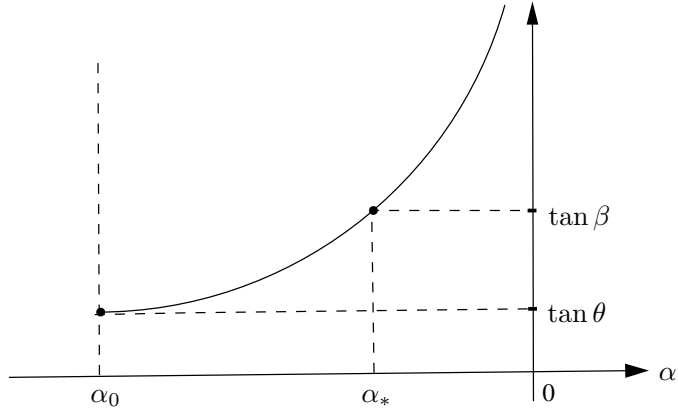}
  \caption{Angle of sectoriality $\beta$. Here $\alpha_0=\arctan(-m_{\infty}(-0))$.}\label{fig-3}
  \end{center}
\end{figure}

\section{Example}

Consider the differential expression with the Bessel potential
$$
l_\nu=-\frac{d^2}{dx^2}+\frac{\nu^2-1/4}{x^2},\;\; x\in [1,\infty)
$$
of order $\nu>0$ in the Hilbert space $\calH=L^2[1,\infty)$.
The minimal symmetric operator
\begin{equation}\label{ex-128}
 \left\{ \begin{array}{l}
 \dA\, y=-y^{\prime\prime}+\frac{\nu^2-1/4}{x^2}y \\
 y(1)=y^{\prime}(1)=0 \\
 \end{array} \right.
\end{equation}
 generated by this expression and boundary conditions has defect numbers $(1,1)$. 
Let $\nu=3/2$. It is known \cite{ABT} that in this case
$$
m_{\infty}(z)= 1-\frac{iz}{\sqrt{z}+i}
$$
and $m_{\infty}(-0)=1.$
The minimal symmetric operator then becomes
\begin{equation}\label{ex-e8-15}
 \left\{ \begin{array}{l}
 \dA\, y=-y^{\prime\prime}+\frac{2}{x^2}y \\
 y(1)=y^{\prime}(1)=0. \\
 \end{array} \right.
\end{equation}
Consider operator $T_h$ of the form \eqref{131} that is written for $h=i$ as 
\begin{equation}\label{ex-135}
 \left\{ \begin{array}{l}
 T_{i}\, y=-y^{\prime\prime}+\frac{2}{x^2}y \\
 y'(1)=i\, y(1) \\
 \end{array} \right..
\end{equation}
This operator $T_i$ will be shared as the main operator by  the family of L-systems realizing functions $(-m_\alpha(z))$ in  \eqref{e-62-psi}-\eqref{e-59-LFT}. It is accretive and $\beta$-sectorial since $\RE h=0>-m_\infty(-0)=-1$ and has the exact angle of sectoriality given by (see \eqref{e10-45})
\begin{equation}\label{e-ex-98}
\tan\beta=\frac{\IM h}{\RE h+m_{\infty}(-0)}=\frac{1}{0+1}=1\quad\textrm{ or}\quad \beta=\frac{\pi}{4}.
\end{equation}
 The family of  L-systems $\Theta_{\tan\alpha, i}$ of the form \eqref{e-64-sys} that realizes functions 
 \begin{equation}\label{e-ex-99}
  -m_\alpha(z)=\frac{({\sqrt{z}-iz+i})\cos\alpha+({\sqrt{z}+i})\sin\alpha}{({\sqrt{z}-iz+i})\sin\alpha-({\sqrt{z}+i})\cos\alpha},
 \end{equation}
  was constructed in \cite{BT18}.
 According to \eqref{e-39-ac-angles} the L-systems  $\Theta_{\tan\alpha, i}$ in \eqref{e-64-sys} are accumulative if
$$
 -1=-m_\infty(-0)\le\tan\alpha\le0.
$$
Applying  part (2) of Theorem \ref{t10-17-i}, we get that the realizing L-system $\Theta_{\tan\alpha, i}$ in \eqref{e-64-sys} is such that the associated operator $\ti\bA_{\tan\alpha,i}$ is extremal accretive if $\mu=\tan\alpha=0$ or $\alpha=0$.
Therefore the L-system
\begin{equation}\label{e-ex-105-sys}
    \Theta_{0, i}= \begin{pmatrix} \bA_{0, i}&K_{0, i}&1\cr \calH_+ \subset
L_2[1,+0) \subset \calH_-& &\dC\cr \end{pmatrix},
\end{equation}
where
\begin{equation}\label{e-ex-84}
\begin{split}
&\bA_{0,i}\, y=-y^{\prime\prime}+\frac{2}{x^2}y-i\,[y^{\prime}(1)-iy(1)]\,\delta^{\prime}(x-1), \\
&\bA^*_{0,i}\, y=-y^{\prime\prime}+\frac{2}{x^2}y+i\,[y^{\prime}(1)+iy(1)]\,\delta^{\prime}(x-1),
\end{split}
\end{equation}
$K_{0, i}{c}=cg_{0, i}$, $(c\in \dC)$ and $g_{0, i}=\delta^{\prime}(x-1)$.
This L-system $\Theta_{0, i}$  realizes the function $-m_{0}(z)=-m_\infty(z)$.
Also,
\begin{equation}\label{e-ex-108-VW}
    \begin{aligned}
    V_{\Theta_{0, i}}(z)&=-m_{0}(z)=-m_\infty(z)=\frac{iz}{\sqrt{z}+i}-1\\
     W_{\Theta_{0, i}}(z)&=-\frac{m_\infty(z)-i}{m_\infty(z)+i}=\frac{(i-1)\sqrt{z}+iz-1-i}{(1+i)\sqrt{z}-iz-1+i}.
     \end{aligned}
\end{equation}
The associate operator $\ti\bA_{0,i}$ is given by \eqref{e-27} as
$$
 \begin{aligned}
    \ti\bA_{0,i}\, y&=-y^{\prime\prime}+\frac{2}{x^2}y-y'(1)\delta(x-1)-y(1)\delta'(x-1) +[y(1)+iy'(1)]\,\delta'(x-1)\\
    &=-y^{\prime\prime}+\frac{2}{x^2}y-y'(1)[\delta(x-1)-i\delta'(x-1)].
    \end{aligned}
$$
The adjoint operator $\ti\bA_{0,i}$ is
$$
\ti\bA_{0,i}^* y=-y^{\prime\prime}+\frac{2}{x^2}y-y'(1)[\delta(x-1)+i\delta'(x-1)],
$$
and consequently
$$
\RE\ti\bA_{0,i}\, y=-y^{\prime\prime}+\frac{2}{x^2}y-y'(1)\delta(x-1)\quad\textrm{ and}\quad \IM \ti\bA_{0,i}\, y=y'(1)\delta'(x-1).
$$
The operator $\ti\bA_{0,i}$ above is accretive according to \cite{BT14} which is also independently confirmed by direct evaluation
$$
(\RE\ti\bA_{0,i}\, y,y)=\|y'(x)\|^2_{L^2}+2\|y(x)/x\|^2_{L^2}\ge0.
$$
Moreover,  according to Theorem \ref{t10-17-i} it is extremal, that is accretive but not $\beta$-sectorial for any $\beta\in(0,\pi/2)$. Indeed,
it is easy to see  that $$(\IM \ti\bA_{0,i}\, y,y)=-|y'(1)|^2,$$  and hence we can have inequality \eqref{e8-29} for all $y\in\calH_+$ only if $\beta=\frac{\pi}{2}$. Thus, this is the case of the extremal operator.
In addition, we have shown that  the  function $-m_{0}(z)=-m_\infty(z)=\frac{iz}{\sqrt{z}+i}-1$ in \eqref{e-ex-108-VW} belongs to the sectorial class $S^{-1,0,\frac{\pi}{2}}$ of inverse Stieltjes functions.


\end{document}